\documentclass[12pt,reqno]{amsart}
\usepackage[left=3cm,top=2cm,right=3cm,bottom=2cm]{geometry}
\usepackage{amssymb}
\usepackage{amsbsy}
\usepackage{epsfig}
\usepackage{tikz}
\usepackage{verbatim}
\usepackage[normalem]{ulem}
\usepackage[pdftex]{hyperref}
\hypersetup{colorlinks=true,linkcolor=blue,citecolor=blue,breaklinks = true}

\tikzstyle{vertex} = [fill,shape=circle,node distance=80pt]
\tikzstyle{edge} = [fill,opacity=.5,fill opacity=.5,line cap=round, line join=round, line width=50pt]
\tikzstyle{elabel} =  [fill,shape=circle,node distance=30pt]

\pgfdeclarelayer{background}
\pgfsetlayers{background,main}

\begin{document}
	\title{The  spectrum of a class of uniform hypergraphs} 	
\author[K. Cardoso]{Kau\^e da Rosa Cardoso} \address{Instituto Federal de Educa\c c\~ao, Ci\^encia e Tecnologia do Rio Grande do Sul - Campus Feliz, CEP 95770-000, Feliz, RS, Brasil} \email{\tt
kaue.cardoso@feliz.ifrs.edu.br}
 	
\author[C. Hoppen]{Carlos Hoppen} \address{Instituto de Matem\'atica e
Estat\'{\i}stica, UFRGS,  CEP 91509--900, Porto Alegre, RS, Brazil}
\email{\tt choppen@ufrgs.br}
 	
\author[V.Trevisan]{Vilmar Trevisan} \address{Instituto de Matem\'atica e
Estat\'{\i}stica, UFRGS,  CEP 91509--900, Porto Alegre, RS, Brazil}
\email{\tt trevisan@mat.ufrgs.br} 

\date{Feb 26, 2019}		

\pdfpagewidth 8.5 in \pdfpageheight 11 in

\newcommand{\h}{\mathcal{H}}
\newcommand{\hk}{\mathcal{H}^k}
\newcommand{\A}{\mathbf{A}}
\newcommand{\x}{\mathbf{x}}
\newcommand{\y}{\mathbf{y}}
\newcommand{\Ah}{\mathbf{A}_\mathcal{H}}
\newcommand{\Ahk}{\mathbf{A}_{\mathcal{H}^k}}
\newcommand{\C}{\mathbb{C}}

	\newtheorem{Pro}{Proposition}
	\newtheorem{Def}{Definition}
	\newtheorem{Teo}{Theorem}
	\newtheorem{Exe}{Example}
    \newtheorem{Lem}[Teo]{Lemma}
    \newtheorem{Cor}[Teo]{Corollary}
    
    \newcommand{\keyword}[1]{\textsf{#1}}

\begin{abstract}
A generalized power hypergraph $\h^k_s$ is obtained from a base hypergraph $\h$ by means of some simple edge-expansion operations. Kang, Liu, Qi and Yuan~\cite{Kang} proved that the nonzero eigenvalues of $\h$ give rise to nonzero eigenvalues of $\h^k_s$. In this paper we show that all nonzero eigenvalues of $\h^k_s$ may be computed from the eigenvalues of its base hypergraph $\h$ and of its subgraphs.  To prove this, we derive spectral results about edge-expansion operations that may be interesting on their own sake.\newline


\end{abstract}

\maketitle

\section{Introduction}

Spectral graph theory analyzes the structure of graphs through the spectrum
of matrices associated with them. This subject is widely studied and has
applications in many areas, such as computer science, chemistry and physics,
in addition to many areas of mathematics (see, for example,
\cite{Cve,Beh,Sak,Shu,Sta}). A spectral theory for the adjacency tensor (or
hypermatrix) associated with a hypergraph has been proposed by Cooper and
Dutle in~\cite{Cooper}. For more information about spectral parameters for
hypergraphs, see~\cite{HuQi, Niki, Pearson, QiBook, Shao}. Cooper and Dutle defined the eigenvalues and eigenvectors of
such a tensor in ways that generalise their graph counterparts (the precise
definitions are in Section~\ref{sec:pre}). Recently the study of hypergraph
spectra has attracted the attention of a large number of
researchers~\cite{Guo,Jin,Yue,Zhang}.

In this paper we are interested in studying the spectrum of a class of uniform hypergraphs that was first considered by Kang, Liu, Qi and Yuan~\cite{Kang}. As usual, a \emph{hypergraph} $\h=(V,E)$ is given by a vertex set $V$ and by a set $E=\{e \colon e \subseteq V\}$ whose elements are called (hyper)edges. A hypergraph is $k$-uniform (or a $k$-graph) if all of its edges have size $k$. For a graph $G$ and an integer $k \geq 2$, the \emph{power graph} $G^k$ is obtained from $G$ by adding $k-2$ new vertices to each edge of $G$. If $v(G)$ and $e(G)$ denote the number of vertices and the number of edges of $G$, respectively, this means that $v(G^k)=v(G)+(k-2)e(G)$ and $e(G^k)=e(G)$. In extremal combinatorics, the graph $G^k$ is often called an \emph{expanded graph} or an \emph{extended graph}, but we shall use the terminology previously used in spectral theory.

The study of the spectrum of $G^k$ began with Hu, Qi and Shao~\cite{Hu} in 2013. Following the work in~\cite{Hu}, there were many efforts to study the spectrum of this type of hypergraph and of generalisations thereof. The literature is mostly devoted to the spectral radius of such hypergraphs, rather than their entire spectrum  (see, for example, \cite{Zhou}, \cite{Khan2}, \cite{Yuan}, \cite{Khan}, \cite{Kang}).

Regarding power graphs, Zhou, Sun, Wang and Bu~\cite{Zhou} proved the following result.
\begin{Teo}\cite[Theorem 16]{Zhou}\label{TeoZhou}
If $\lambda \neq 0$ is an eigenvalue of a graph $G$, then the complex solutions of the equation
$x^k=\lambda^{2}$ are eigenvalues of the power hypergraph $G^k$.
\end{Teo}

It is natural to ask whether the entire spectrum of $G^k$ consists of eigenvalues of this type and, if this is not the case, if the remaining eigenvalues are related to $G$ in some way. The answer to our first question is no, which is illustrated by the power hypergraph $\h=(C_4)^3$: $\sqrt[3]{2}$ is one of its eigenvalues, even though the nonzero eigenvalues of the cycle $C_4$ are only $2$ and $-2$. Of course, the solutions of the equation $x^3=4$ are also nonzero eigenvalues of $\h$.

In this paper, we address the second question and prove that the entire spectrum of $G^k$ may be obtained from $G$ and its subgraphs. In fact, we do this for a more general class of hypergraphs.
				
Let $\h=(V,E)$ be an $r$-graph. Generalizing the notion defined above, given $k \geq r$, we define the \emph{$k$-expansion} $\h^k$ of $\h$ as the $k$-graph obtained by adding $k-r$ new vertices to each edge of $\h$.  Moreover, for an integer $s \geq 1$, we define the \emph{$s$-extension} $\h_s$ of $\h$ as the $sr$-graph obtained by replacing each vertex $v\in V$ by a set $S_v$ of cardinality $s$. Each edge $\{v_1,\ldots, v_r\}$
in $\h$ gives rise to the edge $S_{v_1}\cup\cdots\cup S_{v_r}$ of $\h_s$. Given integers $r \geq 2$, $s\geq 1$ and $k \geq rs$,  a \emph{generalized power hypergraph} is the $k$-graph $\h^k_s =(\h_s)^k$ obtained from an $r$-graph $\h$. For more formal versions of these definitions, we refer the reader to Section~\ref{sec:pre}.

Kang, Liu, Qi and Yuan~\cite{Kang} have studied the spectrum of generalized power hypergraphs.
Even though their original result is stated in a more specific setting, its proof immediately implies the following generalization of Theorem~\ref{TeoZhou}.
\begin{Teo}\cite[Theorem 14]{Kang}\label{TeoKang}
Let $\lambda\neq0$ be an eigenvalue of an $r$-graph $\h$ and consider integers $s\geq 1$ and $k \geq rs$.  The complex solutions of the equation $x^k=\lambda^{rs}$ are eigenvalues of the generalized power hypergraph $\h^k_s$.	
\end{Teo}

The main result in this paper is the full characterization of the spectrum of a generalized power hypergraph.
\begin{Teo}\label{Main}
Let $\h$ be an $r$-graph, fix integers $s \geq 1$ and $k \geq rs$ and consider a constant $\lambda \in \mathbb{C}$ such that $\lambda \neq 0$.
\begin{itemize}
	
\item[(a)] Assume that $k = rs+1$ or that $k=rs$ and $s \geq 2$. The number $\lambda$ is an eigenvalue of $\h^k_s$ if and only if some induced subgraph $\mathcal{G}$ of $\h$ has an eigenvalue $\beta$ such that $\beta^{rs}=\lambda^k$.		
	
\item[(b)] For $k > rs+1$, the number $\lambda$ is an eigenvalue of $\h^k_s$ if and only if some subgraph $\mathcal{G}$ of $\h$ has an eigenvalue $\beta$ such that $\beta^{rs}=\lambda^k$.

\end{itemize}
\end{Teo}
In other words, to find the entire (non-zero) spectrum of $\h^k_s$, it suffices to compute eigenvalues of subgraphs of $\h$ (actually, of induced subgraphs if $k \leq rs+1$). An immediate consequence of this fact is that, whenever $\lambda$ is an eigenvalue of $\h^k_s$ and $\varepsilon$ is a $k$th root of unity, the product $\varepsilon \lambda$ is also an eigenvalue of $\h^k_s$. We should mention that Theorem~\ref{Main} might be stated in an alternative way, whose proof is cleaner. To this end, we say that an eigenpair $(\lambda,\x)$ is \emph{strictly nonzero} if the eigenvalue $\lambda$ is nonzero and all the entries of the eigenvector $\x$ are nonzero. It turns out that, to compute the eigenvalues of $\h^k_s$, it suffices to look for all strictly nonzero eigenpairs of subgraphs (or induced subgraphs) of $\h$, and not for all eigenvalues of all (induced) subgraphs. However, we believe that stating the result with a condition on the components of the eigenvectors would make it harder to apply in some situations.

To prove Theorem~\ref{Main}, we shall proceed as follows. First we observe that an eigenvalue of a generalized power hypergraph whose eigenvector has zero entries forms a nonzero eigenpair of the induced subgraph obtained by removing all vertices associated with the zero entries (see Lemma~\ref{LemEigenpair_hk}). Then we relate strictly nonzero eigenpairs of a generalized power hypergraph with eigenvalues of its base hypergraph (see  Lemma~\ref{Lema_h_hks}).

In our proofs, in addition to edge-expansion operations, we will use vertex-removal and edge-removal operations. If $\h$ is an $r$-graph and $v \in V(\h)$, we let $\h \triangleleft v$ be the hypergraph obtained by removing $v$ and any vertices that become isolated when the edges containing $v$ are deleted. For $e \in E(\h)$, $\h-e$ is the hypergraph obtained by removing $e$ from $E(\h)$ along with any vertices that become isolated when $e$ is removed. The following is a result by Zhou, Sun, Wang and Bu~\cite{Zhou}.
\begin{Teo}\cite[
Theorem 13]{Zhou}\label{TeoEdgeRemocao}
If an $r$-graph $\h$ has an edge $e\in E(\h)$ which contains at least two
vertices of degree one, then the eigenvalues of $\h-e$ are also eigenvalues of
$\h$.
\end{Teo}

To obtain our results, we shall prove the following generalizations of this result.
\begin{Teo}\label{thm_subgraphs}
Let $\h$ be an $r$-graph.

\begin{itemize}
\item[(a)] If $\h$ contains distinct vertices $u$ and $v$ that are contained in exactly the same edges, then the eigenvalues of
$\h\triangleleft v$ are also eigenvalues of $\h$.

\item[(b)] If $\h$ contains a vertex $v$ such that every edge containing $v$ has a vertex with degree 1 other than $v$, then the eigenvalues of $\h\triangleleft v$ are also eigenvalues of $\h$.

\end{itemize}
\end{Teo}

The paper is organized as follows. In Section \ref{sec:pre} we introduce the relevant terminology and prove results about the effect of edge-expansion and extension, and of vertex and edge deletion, on the spectrum of a hypergraph. In particular, we prove Theorem~\ref{thm_subgraphs}. The proof of Theorem~\ref{Main} is the subject of Section \ref{sec:main}.

\section{Preliminaries}\label{sec:pre}

In this section, we shall present some basic definitions about hypergraphs and
tensors, as well as terminology, notation and concepts that will be useful in our proofs. More details can be found in \cite{Cooper}, \cite{Hu} and \cite{Shao}.
\begin{Def}
A tensor (or hypermatrix) $\A$ of dimension $n$ and order $r$ is a collection of $n^r$
elements $a_{i_1 \dots i_r}\in\mathbb{C}$ where
$i_1,\dots,i_r\in[n] =\{1,2,\ldots,n\}$.
\end{Def}

Let $\A$ be a tensor of dimension $n$ and order $r$ and  $\mathbf{B}$ be a
tensor of dimension $n$ and order $s$. We define the product of  $\A$ by
$\mathbf{B}$ as a tensor $\mathbf{C}=\A\mathbf{B}$ of dimension $n$ and
order $(r-1)(s-1)+1$, where	
\[c_{j\alpha_1\cdots\alpha_{r-1}}=\sum_{i_2,\dots, i_r = 1}^{n}a_{ji_2\dots i_r}b_{i_2\alpha_1}\cdots b_{i_r\alpha_{r-1}}\;\; \emph{with} \;\; j\in[n]\;\; \emph{and} \;\;\alpha_1,\dots,\alpha_{r-1}\in[n]^{s-1}.\]
	
In particular, if $\x\in \mathbb{C}^n$ is a vector, then the $i$th component of $\A\x$ is given by
\[(\A \x)_i=\sum_{i_2,\dots, i_r = 1}^{n}a_{ii_2\dots i_r}x_{i_2}\cdots x_{i_r}\quad\forall i\in [n].\]

If $\x=(x_1,\cdots,x_n)\in \mathbb{C}^n$ is a vector and $r$ is a positive integer, we let $\x^{[r]}$ denote the vector in $\mathbb{C}^n$ whose $i$th component is given by $x_i^r$.
\begin{Def}
A number $\lambda \in \mathbb{C}$ is an eigenvalue of a tensor $\A$ of dimension $n$ and order $r$ if
there is a nonzero vector $\x\in \mathbb{C}^n$ such that
	\[\A \x=\lambda \x^{[r-1]}.\]
We say that $\x$ is an eigenvector of $\A$ associated with the eigenvalue $\lambda$ and that
$(\lambda, \x)$ is an eigenpair of $\A$.
\end{Def}
Recall that an eigenpair $(\lambda,\x)$ is \emph{strictly nonzero} if the eigenvalue $\lambda$ is nonzero and all the entries of the eigenvector $\x$ are nonzero.

As usual, if $\h$ is an $r$-graph, a hypergraph $\mathcal{F}$ is a \emph{subgraph} of $\h$
if $V(\mathcal{F})\subseteq V(\h)$ and $E(\mathcal{F})\subseteq E(\h)$. For
$S\subset V$, we say that $\h[S]=(S,E(\h) \cap 2^{S})$ is the subgraph of
$\h$ \emph{induced} by $S$, where $2^S$ denotes the power set of $S$. The \emph{degree} of a vertex $v\in V$ is the number of
edges that contain $v$. Vertices of degree zero are said to be \emph{isolated} vertex.

\begin{Def}
Let $\h$ be an $r$-graph with $n$ vertices. The adjacency
tensor $A_\h$ of $\h$ is the tensor of dimension $n$ and order $r$ such that
\[
a_{i_1 \dots i_r}= \begin{cases} \frac{1}{(r-1)!},\quad \emph{if}\;
\{i_1,\dots,i_r\}\in E(\h), \\ \quad 0, \quad \;\;\; \emph{otherwise.}
\end{cases}\]
\end{Def}

If $\alpha=\{i_1,\dots,i_{s}\}$ is a set of integers, we denote $x^\alpha=x_{i_1}\dots x_{i_{s}}$. For an $r$-graph $\h$ and $i \in V(\h)$, we consider
\[E(i)=E_\h(i)=\{e-\{i\}:i\in e\in E(\h)\}.\]
By definition, the $i$th component of the adjacency tensor $\Ah$ of $\h$ is given by
\begin{equation}\label{eq_adjacency}
(\Ah\x)_i=\sum_{i_2,\dots, i_r =
1}^{n}\frac{1}{(r-1)!}x_{i_2}\cdots x_{i_r}=
  \sum_{\alpha\in E(i)}x^\alpha.
 \end{equation}
Therefore a pair $(\lambda,\x)$ is an eigenpair of $\h$, where $\x=(x_1,\ldots,x_n)$, if and only if the following system of equations is satisfied:
 \begin{equation}\label{system}
 \sum_{\alpha\in E(i)}x^\alpha=\lambda x_i^{r-1}\quad \forall i\in V(\h).
 \end{equation}

\begin{Def}
Let $\h=(V,E)$ be an $r$-graph and let $k \geq r$. We define the $k$-expansion of $\h$ as the $k$-graph $\h^k$, obtained by adding $k-r$ vertices of degree one
on each edge of $\h$. More precisely, we define sets $S_e=\{v^1_e, \ldots,v^{k-r}_e\}$ for each $e \in E$ such that $S_e \cap V=S_e \cap S_f=\emptyset$ for all $e,f \in E$, $e \neq f$. The sets of vertices and edges of $\h^k$ are
\[V(\h^k)=V \cup \bigcup_{e\in E} S_e \quad \emph{and}\quad E(\h^k)=\{e\cup S_e: \forall e \in E\}.\]	
\end{Def}

We say that the vertices of $V(\h)$ are the \emph{main vertices} of $\h^k$, while the vertices in some $S_e$ are the \emph{additional vertices}. In particular, $v(\h^k)=v(\h)+(k-r)e(\h)$ and $e(\h^k)=e(\h)$.

\begin{Exe} The $6$-expansion $(P_4)^6$ of the path $P_4$ is depicted in Figure \ref{fig:ex1}.
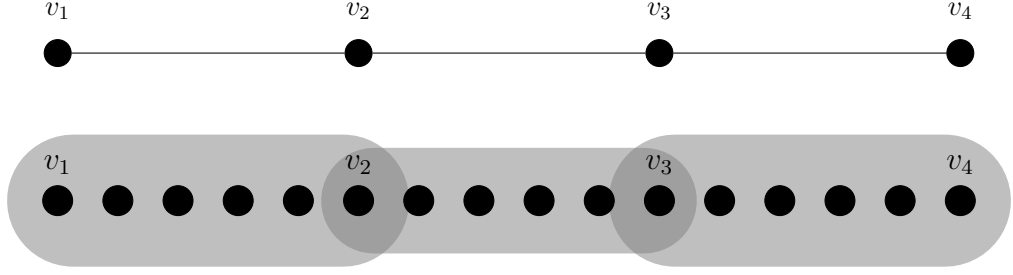
\begin{figure}[h]\label{fig:ex1}
  \centering
	\begin{tikzpicture}
	[scale=1,auto=left,every node/.style={circle,scale=0.9}]
	\node[draw,circle,fill=black,label=below:,label=above:\(v_1\)] (v1) at (0,0) {};
	\node[draw,circle,fill=black,label=below:,label=above:\(v_2\)] (v2) at (4,0) {};
	\node[draw,circle,fill=black,label=below:,label=above:\(v_3\)] (v3) at (8,0) {};
	\node[draw,circle,fill=black,label=below:,label=above:\(v_4\)] (v4) at (12,0) {};
	\path
	(v1) edge node[left]{} (v2)
	(v2) edge node[below]{} (v3)
	(v3) edge node[left]{} (v4);
	\end{tikzpicture}

	\vspace{0.5cm}
	
	\begin{tikzpicture}
	\node[draw,circle,fill=black,label=below:,label=above:\(v_1\)] (v1) at (0,0) {};
	\node[draw,circle,fill=black,label=below:,label=above:] (v21) at (0.8,0) {};
	\node[draw,circle,fill=black,label=below:,label=above:] (v22) at (1.6,0) {};
	\node[draw,circle,fill=black,label=below:,label=above:] (v23) at (2.4,0) {};
	\node[draw,circle,fill=black,label=below:,label=above:] (v24) at (3.2,0) {};
	\node[draw,circle,fill=black,label=below:,label=above:\(v_2\)] (v3) at (4,0) {};
	\node[draw,circle,fill=black,label=below:,label=above:] (v41) at (4.8,0) {};
	\node[draw,circle,fill=black,label=below:,label=above:] (v42) at (5.6,0) {};
	\node[draw,circle,fill=black,label=below:,label=above:] (v43) at (6.4,0) {};
	\node[draw,circle,fill=black,label=below:,label=above:] (v44) at (7.2,0) {};
	\node[draw,circle,fill=black,label=below:,label=above:\(v_3\)] (v5) at (8,0) {};
	\node[draw,circle,fill=black,label=below:,label=above:] (v61) at (8.8,0) {};
	\node[draw,circle,fill=black,label=below:,label=above:] (v62) at (9.6,0) {};
	\node[draw,circle,fill=black,label=below:,label=above:] (v63) at (10.4,0) {};
	\node[draw,circle,fill=black,label=below:,label=above:] (v64) at (11.2,0) {};
	\node[draw,circle,fill=black,label=below:,label=above:\(v_4\)] (v7) at (12,0) {};

	\begin{pgfonlayer}{background}
	\draw[edge,color=gray] (v1) -- (v3);
	\draw[edge,color=gray,line width=40pt] (v3) -- (v5);
	\draw[edge,color=gray] (v5) -- (v7);
	
	\end{pgfonlayer}
	
	\end{tikzpicture}
\caption{$P_4$ and its $6$-expansion $(P_4)^6$.}
\end{figure}
\end{Exe}

\begin{Def}				
Let $\h=(V,E)$ be an $r$-graph and let $s \geq 1$. We define the $s$-extension $\h_s$ of $\h$ as the $sr$-graph obtained by replacing each vertex $v_i \in V$
by a set $S_{v_i}=\{v_{i1}, \ldots, v_{is}\}$, where $S_v \cap S_w = \emptyset$ for $v \neq w$. More precisely, the sets of vertices and edges of $\h_s$ are
\[V(\h_s)=\bigcup_{v\in V} S_v \quad \emph{and} \quad E(\h_s)=\{S_{v_1} \cup \cdots \cup S_{v_r} \colon e=\{v_1,\ldots, v_r\} \in E\}\]
\end{Def}	

For simplicity, we identify one vertex in each $S_v$ with the original vertex $v$ and call it a \emph{main vertex}, while the remaining vertices are again called \emph{copies}. In particular, $v(\h_s)=s \cdot v(\h)$ and $e(\h_s)=e(\h)$.

\begin{Exe}
Figure \ref{fig:ex2} depicts the $3$-extension $(P_4)_3$ of $P_4$.

\begin{figure}[h]
  \centering
	\begin{tikzpicture}
	\node[draw,circle,fill=black,label=below:,label=above:\(v_{11}\)] (v11) at (0,0) {};
	\node[draw,circle,fill=black,label=below:,label=above:\(v_{12}\)] (v12) at (1,0) {};
	\node[draw,circle,fill=black,label=below:,label=above:\(v_{13}\)] (v13) at (2,0) {};
	\node[draw,circle,fill=black,label=below:,label=above:\(v_{21}\)] (v21) at (3.33,0) {};
	\node[draw,circle,fill=black,label=below:,label=above:\(v_{22}\)] (v22) at (4.33,0) {};
	\node[draw,circle,fill=black,label=below:,label=above:\(v_{23}\)] (v23) at (5.29,0) {};
	\node[draw,circle,fill=black,label=below:,label=above:\(v_{31}\)] (v31) at (6.72,0) {};
	\node[draw,circle,fill=black,label=below:,label=above:\(v_{32}\)] (v32) at (7.67,0) {};
	\node[draw,circle,fill=black,label=below:,label=above:\(v_{33}\)] (v33) at (8.67,0) {};
	\node[draw,circle,fill=black,label=below:,label=above:\(v_{41}\)] (v41) at (10,0) {};
	\node[draw,circle,fill=black,label=below:,label=above:\(v_{42}\)] (v42) at (11,0) {};
	\node[draw,circle,fill=black,label=below:,label=above:\(v_{43}\)] (v43) at (12,0) {};

	\begin{pgfonlayer}{background}
	\draw[edge,color=gray] (v11) -- (v23);
	\draw[edge,color=gray,line width=40pt] (v21) -- (v33);
	\draw[edge,color=gray] (v31) -- (v43);
	
	\end{pgfonlayer}
	\end{tikzpicture}
\label{fig:ex2}\caption{The 3-extension $(P_4)_3$ of $P_4$.}
\end{figure}
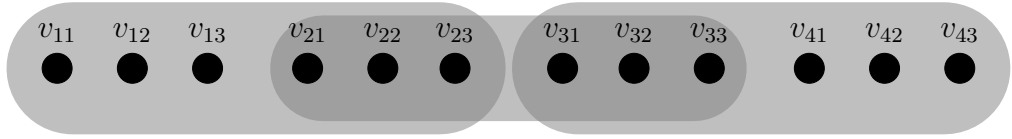
\end{Exe}

We are now ready to describe the class of  \emph{generalized power hypergraphs}, which was introduced in~ \cite{Kang}. Each of its elements may be obtained as follows. Fix integers $r \geq 2$, $s \geq 1$ and $k \geq rs$, and consider an $r$-graph $\h$. The generalized power hypergraph $\h^k_s$ is the $k$-graph $(\h_s)^k$. We say that $\h$ is the \emph{base hypergraph}.

\begin{Exe} Figure \ref{fig:ex3} depicts the generalized power hypergraph $(P_4)_2^6$.

\begin{figure}[h]
  \centering
	\begin{tikzpicture}
	\node[draw,circle,fill=black,label=below:,label=above:\(v_{11}\)] (v1) at (0,0) {};
	\node[draw,circle,fill=black,label=below:,label=above:\(v_{12}\)] (v11) at (1,0) {};
	\node[draw,circle,fill=black,label=below:,label=above:] (v22) at (1.84,0) {};
	\node[draw,circle,fill=black,label=below:,label=above:] (v23) at (2.64,0) {};
	\node[draw,circle,fill=black,label=below:,label=above:\(v_{21}\)] (v3) at (3.5,0) {};
	\node[draw,circle,fill=black,label=below:,label=above:\(v_{22}\)] (v31) at (4.5,0) {};
	\node[draw,circle,fill=black,label=below:,label=above:] (v42) at (5.42,0) {};
	\node[draw,circle,fill=black,label=below:,label=above:] (v44) at (6.58,0) {};
	\node[draw,circle,fill=black,label=below:,label=above:\(v_{31}\)] (v5) at (7.5,0) {};
	\node[draw,circle,fill=black,label=below:,label=above:\(v_{32}\)] (v51) at (8.5,0) {};
	\node[draw,circle,fill=black,label=below:,label=above:] (v62) at (9.34,0) {};
	\node[draw,circle,fill=black,label=below:,label=above:] (v63) at (10.14,0) {};
	\node[draw,circle,fill=black,label=below:,label=above:\(v_{41}\)] (v64) at (11,0) {};
	\node[draw,circle,fill=black,label=below:,label=above:\(v_{42}\)] (v7) at (12,0) {};

	\begin{pgfonlayer}{background}
	\draw[edge,color=gray] (v1) -- (v31);
	\draw[edge,color=gray,line width=40pt] (v3) -- (v51);
	\draw[edge,color=gray] (v5) -- (v7);
	
	\end{pgfonlayer}
	
	\end{tikzpicture}

  \caption{The generalized power hypergraph $(P_4)^6_2$.}
  \label{fig:ex3}
\end{figure}
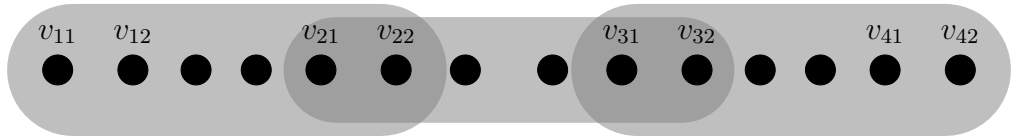		

\end{Exe}

\begin{Def}
Let $\h$ be a $k$-graph. For $I\subset V(\h)$ we define $\h \triangleleft I$
as the subgraph of $\h$ obtained by removing all vertices of $I$,
all edges that have vertices in $I$, and all vertices that become isolated
after the removal of these edges.
\end{Def}

\begin{Def}
Let $\h$ be a $k$-graph. For $A\subset E(\h)$ we define $\h -A$ as the
subgraph of $\h$ obtained by removing all edges of $A$ and all
vertices that become isolated after the removal of these edges.
\end{Def}

To differentiate the edges of $\h$ and $\h^k_s$, when necessary, the edges of
$\h$ will be called $r$-edges and the edges of $\h^k_s$ will be called
$k$-edges. For a subset $A\subset E(\h)$, we denote by $A^k_s$ the subset of
$k$-edges in $\h^k_s$ obtained from the $r$-edges of $A$.

The following lemma summarizes simple facts about these operations.

\begin{Lem}\label{LemaPowerRemocao}
Let $\h$ be an $r$-graph for some $r \geq 2$ and let $s\geq 1$ and $k \geq rs$ be integers.

\begin{itemize}

\item[(a)] If $e \in E(\h)$ contains a vertex $v$ of
degree one, then $\h\triangleleft v = \h - e$.

\item[(b)] If $A\subset E(\h)$, then $(\h-A)^k_s=\h^k_s-A^k_s$.

\item[(d)] If $I\subset V(\h)$, then $(\h\triangleleft
I)^k_s=\h^k_s\triangleleft I.$
\end{itemize}
\end{Lem}

\begin{proof}
To prove part (a), note that $\h\triangleleft v$ is obtained from $\h$ by removing $v$, the edge $e$ and any vertices that become isolated when $e$ is removed. Since $v$ has degree 1, it is one of the isolated vertices that are removed from $\h$ to produce $\h-e$, so that $\h\triangleleft v$ and $\h-e$ are the same hypergraph. To prove part (b), one may prove separately that $(\h-A)^k=\h^k-A^k$ and that $(\h-A)_s=\h_s-A_s$, and combine the two results. Part (c) may be easily obtained with this same strategy.
\end{proof}

We are now ready to establish some spectral properties involving these operations, as well as expansions and extensions. Before doing this, we prove Theorem~\ref{thm_subgraphs}, which generalizes~\cite[Theorem 13]{Zhou}.
\begin{proof}[Proof of Theorem~\ref{thm_subgraphs}] We start with part (a). Let $\h$ be an $r$-graph and assume that $v$ and $u$ are two vertices that are contained in exactly the same edges. Let $(\lambda,\x)$ be an eigenpair of $\h\triangleleft v$. Define a vector $\mathbf{y}$ of dimension $n=v(\h)$ by $y_i=x_i$ if $i\in
V(\h\triangleleft v)$ and $y_i=0$ otherwise. We see that
\[(\Ah\mathbf{y})_v = \sum_{\alpha\in E(v)}y^{\alpha}=\sum_{\alpha\in E(v)}y_{u}y^{\alpha-\{u\}} =
\sum_{\alpha\in E(v)}0y^{\alpha-\{u\}} =0=\lambda y_v^{r-1}.\]

Consider a vertex $w \in V(\h)-\{v\}$ such that $ w \not \in \h\triangleleft
v$. We have
\[(\Ah\mathbf{y})_w =\sum_{\alpha\in
E(w)}y^{\alpha}=\sum_{\alpha\in E(w)}y_{v}y^{\alpha-\{v\}} =   \sum_{\alpha\in
E(w)}0y^{\alpha-\{v\}} =0=\lambda y_w^{r-1}.\]

Finally, consider a vertex $w \in V(\h\triangleleft v)$. Let $e_1,\ldots,e_p$ be
the edges of $\h\triangleleft v$ that contain $w$ and let $f_{1},\ldots,f_{q}$
be the edges of $\h$ that contain $v$ and $w$. We have
\begin{eqnarray*}
(\Ah\mathbf{y})_w&=&\sum_{\alpha\in E(w)}y^{\alpha}=\sum_{i=1}^qy^{f_i-\{w\}}+\sum_{j=1}^py^{e_j-\{w\}}=0+\sum_{j=1}^px^{e_j-\{w\}}\\
&=&(\mathbf{A}_{\h \triangleleft v}\mathbf{x})_w=\lambda x_w^{r-1}=\lambda y_w^{r-1},
\end{eqnarray*}
as required.

For part (b), note that the result follows immediately from (a) if $v$ has degree 1, as the other vertex of degree 1 in the edge $e$ containing $v$, call it $u$, lies in the same edges as $v$. So we may assume that $v$ has degree at least two. Let $(\lambda,\x)$ be an eigenpair of $\h\triangleleft v$.
Define a vector $\mathbf{y}$ of dimension $n=v(\h)$ by $y_i=x_i$ if $i\in
V(\h\triangleleft v)$, and $y_i=0$ otherwise.

For each $\alpha\in E(v)$, let $w_{\alpha}\in \alpha$ be a vertex of degree one given by the hypothesis. Note that \[(\Ah\mathbf{y})_v =\sum_{\alpha\in
E(v)}y^{\alpha}=\sum_{\alpha\in E(v)} y_{w_{\alpha}}y^{\alpha-\{w_{\alpha}\}} =
\sum_{\alpha\in E(v)}0y^{\alpha-\{w_{\alpha}\}} =0=\lambda y_v^{r-1}.\]

Next suppose that $w$ is a vertex of $\h$ that becomes isolated when $v$ is removed, so that all edges in $E(\h)$ containing $w$ also contain $v$. We have
\[(\Ah\mathbf{y})_w=\sum_{e \in E(w)}y^{e-\{w\}}=\sum_{e \in E(w)} y_vy^{e-\{w,v\}}=\sum_{e \in E(w)} 0y^{e-\{w,v\}}=0=\lambda y_w^{r-1}.\]

Finally, consider a vertex $u$ in $\h\triangleleft v$. Let
$e_1,\ldots,e_p$ be the edges in $\h\triangleleft v$ that contain $u$ and
$f_{1},\ldots,f_{q}$ the edges that contain
$u$ and $v$. We have
\[(\Ah\mathbf{y})_u=\sum_{\alpha\in
E(u)}y^{\alpha}=\sum_{i=1}^qy^{f_i-\{u\}}+\sum_{j=1}^py^{e_j-\{u\}}=0+\sum_{j=1}^px^{e_j-\{u\}}=
\lambda x_u^{r-1} =\lambda y_u^{r-1}.\]
This shows that $(\lambda,\y)$ is an
eigenpair of $\h$.
\end{proof}

The following are immediate consequences of Theorem~\ref{thm_subgraphs}.
\begin{Cor}\label{Cor_hs}
Let $\h$ be an $r$-graph, $s \geq 2$ and $I \subset V(\h_s)$. If  $V(\h_s\triangleleft I) \neq \emptyset$, then the eigenvalues of $\h_s\triangleleft I$ are eigenvalues of $\h_s$.
\end{Cor}

\begin{proof}
To obtain $\h_s\triangleleft I$, we may order the vertices in $I$ arbitrarily and remove them one by one in this order (of course, nothing is done if a vertex became of degree 1 in the previous step). The result follows by applying Theorem~\ref{thm_subgraphs}(a) as any vertex of $\h_s$ lies in the same edges as the other vertices that lie in the same set $S_v$. Note that we never delete more than one vertex in some $S_v$, as the other vertices become isolated when we remove the first vertex of $S_v$.
\end{proof}

\begin{Cor}\label{Cor_hr}
Let $\h$ be an $r$-graph and $I \subset V(\h)$. If  $V(\h^{r+1}\triangleleft I) \neq \emptyset$, then the eigenvalues of
$\h^{r+1}\triangleleft I$ are eigenvalues of $\h^{r+1}$.	
\end{Cor}

\begin{proof} Since all vertices in $I$ are main vertices, all edges of $\h^{r+1}$ have a degree one vertex that is not in $I$. The result follows immediately from Theorem~\ref{thm_subgraphs}(b) by removing the vertices in $I$ one by one.
\end{proof}

\begin{Cor}\label{Cor_hk}
Let $\h$ be an $r$-graph and $I \subset V(\h^k)$ for $k > r+1$. If  $V(\h^k\triangleleft I) \neq \emptyset$, then the
eigenvalues of $\h^k\triangleleft I$ are eigenvalues of $\h^k$.
\end{Cor}

\begin{proof} Since $k>r+1$, it holds that every edge contains at least two vertices of degree 1, and hence Theorem~\ref{thm_subgraphs}(a) always applies.
\end{proof}

Note that, in Corollary~\ref{Cor_hr}, we may only remove main vertices. In
general, we cannot remove additional vertices because we do not know whether
there are other vertices of degree one in the edge. There is no such restriction in Corollary~\ref{Cor_hk}, since we are adding at least two vertices of degree one in each
edge.

To conclude this section, we state a simple fact that may be viewed as a converse to the above corollaries.
\begin{Lem}\label{LemEigenpair_hk} Let $\h$ be an $r$-graph. If $(\lambda,\x)$ is an eigenpair
	of $\h$ with $\lambda\neq 0$, then there is a strictly nonzero eigenpair $(\lambda,\y)$  of
	$\h \triangleleft I$ for some $I\subset V(\h)$.
\end{Lem}

\begin{proof} If $(\lambda,\x)$ is a strictly nonzero eigenpair, then we
	consider $I=\emptyset$ and $\h \triangleleft I = \h$. Otherwise, consider the nonempty set
	$I=\{u\in V(\h) \colon x_u=0\}$ and define a vector $\mathbf{y}$ of dimension $m=v(\h \triangleleft I)$ with $y_u=x_u$ if $u\in
	V(\h \triangleleft I)$. Clearly, all entries of $\mathbf{y}$ are nonzero, and we claim that $V(\h \triangleleft I) \neq \emptyset$. To see why this is true, first note that $I\neq V$, otherwise $\x= 0$ would not be an eigenvector. If $V(\h \triangleleft I) = \emptyset$, any vertex $v$ in $V-I$ must removed together with the vertices of $I$, that is, for each edge $e$ containing $v$ there is a vertex $u_e$ in this edge such that $u_e\in I$. This implies that, with respect to $v$, equation~\eqref{system} becomes	
$$\lambda x_v^{r-1}=\sum_{\alpha\in E(v)}x^\alpha=0,$$
which would lead to $\lambda=0$, a contradiction. This proves that $V(\h \triangleleft I) \neq \emptyset$.

	Let $u \in V(\h \triangleleft I)$ and $e_1,\ldots,e_p \in E(\h)$
	be the edges that contain $u$ and some vertex of $I$ and let
	$f_1,\ldots,f_q \in E(\h)$ be the other edges that contain $u$. We have
	\[(\A_{\h \triangleleft I}\mathbf{y})_u=\sum_{\alpha\in
		E_{\h \triangleleft I}(u)}y^{\alpha}= \sum_{i=1}^qy^{f_i-\{u\}}=
	\underbrace{\sum_{i=1}^px^{e_i-\{u\}}}_{= \; 0}+\sum_{i=1}^qx^{f_i-\{u\}}=\lambda
	x_u^{k-1}=\lambda y_u^{k-1},\] and the result follows. 	
\end{proof}

\section{The spectrum of generalized power hypergraphs}\label{sec:main}
In this final section, we shall prove Theorem~\ref{Main}. We start with algebraic properties of the eigenvectors of an $r$-graph $\h$ and of its expansions and extensions.

\begin{Lem}\label{LemaTwoVertices}
	Let $\h$ be an $r$-graph. If $\h$ has two vertices $v_1$ and $v_2$ that lie in the same edges, then for each nonzero eigenvalue
	$\lambda$ of $\h$, the coordinates $x_{v_1}$ and $x_{v_2}$ of an eigenvector $\x$ associated with $\lambda$ differ only by a multiple of an $r$-th root of unity. More precisely,
	$x_{v_1}=\varepsilon x_{v_2}$ with $\varepsilon^r=1$.		
\end{Lem}

\begin{proof} Notice that the eigenvalue equations for $v_1$ and $v_2$ are
	$$(\A_{\h}\x)_{v_1}= \sum_{\alpha\in E(v_1)}x^{\alpha}=\lambda x_{v_1}^{r-1}
	\quad\textrm{and}\quad
	(\A_{\h}\x)_{v_2}= \sum_{\alpha\in E(v_2)}x^{\alpha}=\lambda x_{v_2}^{r-1}.$$
	Multiplying the first equation by $x_{v_1}$ and the second by
	$x_{v_2}$ we have
	\[\lambda x_{v_1}^{r}=\sum_{e\in E(\h)| v_1\in
		e}x^{e}=\sum_{e\in E(\h)| v_2\in e}x^{e}=\lambda x_{v_2}^{r}.\]
	
	As $\lambda\neq 0$, we conclude that $x_{v_1}^r=x_{v_2}^r$, so that $x_{v_1}=\varepsilon x_{v_2}$ for some
	$\varepsilon^r=1$.
\end{proof}

We introduce the following notation. A $k$-edge of $\h^k_s$ is denoted $\textbf{e}=S_{v_1} \cup \cdots \cup S_{v_r} \cup S_e$, where
$e=\{v_{1},\ldots,v_{r}\}\in E(\h)$, $S_{v_i}= \{v^i_1,\ldots, v^i_s\}$ for $i \in [r]$  and
$S_e=\{u^1_e,\ldots,u^{k-sr}_e\}$.  Let $(\lambda,\mathbf{x})$ be an eigenpair of  $\h^k_s$. By Lemma~\ref{LemaTwoVertices}, we know that, for all $e \in E(\h)$ and $i \in [k-rs]$, we have $x_{u^i_e}= \omega_{e,i} x_{u^1_e}$, where $\omega_{e,i}^k=1$. Also by Lemma~\ref{LemaTwoVertices}, given $i \in [r]$ and $p \in [s]$, we have $x_{v^i_p}= \varepsilon_{i,p} x_{v^i_1}$, where $\varepsilon_{i,p}^k=1$.

With the notation $\varepsilon_{S_{v_i}}=
\prod_{j=1}^{s}\varepsilon_{i,j}$, $\omega_{S_e}=\prod_{j=1}^{k-rs}\omega_{e,j}$ and
$\epsilon_\textbf{e}=\varepsilon_{S_{v_1}}\cdots
\varepsilon_{S_{v_r}}\omega_{S_e}$, we may write
\begin{eqnarray}
x_{S_{v_i}}&=&x_{v^i_{1}}x_{v^i_{2}}\cdots x_{v^i_{s}}=x_{v^{i}_1}\varepsilon_{i,2}x_{v^{i}_1}\cdots \varepsilon_{i,s}x_{v^{i}_1}=\varepsilon_{S_{v_i}} x_{v^{i}_1}^s \label{eq_aux1}\\
x_{S_e}&=&x_{u^1_{e}}x_{u^2_{e}}\cdots x_{u^{k-rs}_{e}}=x_{u^1_{e}}\omega_{e,2}x_{u^1_{e}}\cdots \omega_{e,k-rs}x_{u^{1}_{e}}=\omega_{S_e} x_{u^1_{e}}^{k-rs} \label{eq_aux2}\\
x^\textbf{e}&=&x_{S_{v_1}}\cdots x_{S_{v_r}}x_{S_e}=\epsilon_\textbf{e} x_{v^{1}_1}^s\cdots x_{v^{r}_1}^sx_{u^1_{e}}^{k-rs}=\epsilon_\textbf{e} (x^e)^sx_{u^1_{e}}^{k-rs}\label{eq_aux3}
\end{eqnarray}

\begin{Lem}\label{LemaAdditionalVertex}
	Let $\h$ be an $r$-graph and $(\lambda,\x)$ be an eigenpair of the generalized power hypergraph $\h^k_s$. If $\lambda \not = 0$, then the entry of $\x$ associated with an additional vertex $u$ satisfies
	$x_{u}^{rs}=\varepsilon \lambda^{-1}(x^e)^s$, where $\varepsilon^k=1$ and
	$e\in E(\h)$ is the $r$-edge that originated the $k$-edge $\textbf{e}$
	containing $u$.
\end{Lem}

\begin{proof} Let $(\lambda,\x)$ be an eigenpair of the generalized power hypergraph $\h^k_s$ such that $\lambda\neq 0$ and fix an additional vertex $u$. Let $e\in E(\h)$ be the $r$-edge of $\h$ that originated the $k$-edge $\mathbf{e}$ of $\h^k_s$
	containing $u$. First consider the case when $u=u_e^1$. Then we have
	\[\lambda x_{u_e^1}^{k-1} = (\A_{\h^k_s})_{u_e^1} = x^{\textbf{e}-\{u_e^1\}} \stackrel{\eqref{eq_aux3}}{=}\epsilon_\textbf{e} (x^e)^sx_{u_e^1}^{k-rs-1} \quad\Rightarrow \quad \epsilon_\textbf{e}\lambda^{-1}(x^e)^s=x_{u_e^1}^{rs}.\]
	
	This implies that the result holds for $u_e^1$, as $\epsilon_\textbf{e}$ is a $k$-root of unity. For a general additional vertex $u$, the result follows by Lemma~\ref{LemaTwoVertices}, as $x_u$ and $x_{u_e^1}$ may differ only by a $k$th root of unity.
\end{proof}

An immediate but useful consequence of this result will be stated as the following corollary.

\begin{Cor}\label{cor_useful}
	Let $\h$ be an $r$-graph and $(\lambda,\x)$ be an eigenpair of the generalized power hypergraph $\h^k_s$ such that $\lambda \neq 0$. If $x_v=0$ for a main vertex, then $x_u=0$ for any additional vertex $u$ contained in some edge that contains $v$.
\end{Cor}

Another algebraic relation between the spectra of $\h$ and $\h^k_s$ is the following.
\begin{Lem}\label{Lema_h_hks}
Let $\h$ be an $r$-graph. If $(\lambda,\x)$ is a strictly nonzero eigenpair
of $\h^k_s$, then there exists an eigenpair $(\beta,\y)$ of $\h$ such that
$\beta^{rs}=\lambda^{k}$ and $y_v^r=x_v^k$, for each vertex
in $v\in V(\h)$.
\end{Lem}

\begin{proof}
Let $u$ be a main vertex in $V(\h^k_s)$. Then

$$\lambda x_u^{k-1} = \sum_{\alpha\in E_{\h^k_s}(u)}x^{\alpha}\stackrel{\eqref{eq_aux3}}{=} \sum_{e\in E(\h)| u\in e} \epsilon_\textbf{e}(x_u)^{s-1}(x^{e-\{u\}})^s(x_{u_{e}^1})^{k-rs}$$

In equation \eqref{Eq_raiz} we will choose the $rs$-th root that preserves equality.
\begin{eqnarray}
\lambda x_u^{k-1}&=& \sum_{e\in E(\h)| u\in e}\left( \left(  \epsilon_\textbf{e}x_u^{s-1}(x^{e-\{u\}})^sx_{u_{e}^1}^{k-rs}\right)^{rs} \right)^{\frac{1}{rs}} \label{Eq_raiz} \\
&=& \sum_{e\in E(\h)| u\in e}\left(\epsilon_\textbf{e}^{rs}x_u^{rs(s-1)}(x^{e-\{u\}})^{rs^2}(x_{u_{e}^1}^{rs})^{k-rs} \right)^{\frac{1}{rs}} \notag\\
&=& \sum_{e\in E(\h)| u\in e}\left(\epsilon_\textbf{e}^{rs}x_u^{rs(s-1)}(x^{e-\{u\}})^{rs^2}(\epsilon_\textbf{e}\lambda^{-1}(x^e)^s)^{k-rs} \right)^{\frac{1}{rs}} 	\notag\\
&=& \sum_ {e\in E(\h)| u\in e}\left( \epsilon_\textbf{e}^k x_u^{s(k-r)}(x^{e-\{u\}})^{ks}\lambda^{-(k-rs)}\right)^{\frac{1}{rs}} 	\notag\\
&=& \left( \lambda^{-(k-rs)}x_u^{s(k-r)}\right)^{\frac{1}{rs}}\sum_ {e\in E(\h)| u\in e}\left( (x^{e-\{u\}})^{ks}\right)^{\frac{1}{rs}} \notag
\end{eqnarray}
This implies that
$$\lambda^\frac{k}{rs} x_u^{\frac{k(r-1)}{r}}\!\!\!\!=\sum_ {e\in E(\h)| u\in e}(x^{e-\{u\}})^\frac{k}{r}.$$

Therefore, $\beta = \lambda^\frac{k}{rs}$ is an eigenvalue of $\h$ with an eigenvector $\y$ given by  $y_v=(x_v)^{\frac{k}{r}}$, for all $v \in V(\h)$.
\begin{comment}
For some solution of these equations $\beta^{rs}=\lambda^{k}$ $y_u^r=(x_u)^{k}$, we
have
\[\sum_ {\alpha\in E_\h(u)}y^\alpha = \beta y_u^{r-1} \quad \forall u \in V(\h)\]
Therefore, we conclude that $(\beta,\y)$ is an eigenpair of $\h$.
\end{comment}
\end{proof}

Next, we prove three particular cases that will imply Theorem~\ref{Main} in its full generality.
\begin{Lem}\label{LemaSubgraph_hs}
Let $\h$ be an $r$-graph and $s\geq2$ be an integer. The pair $(\lambda,\x)$ is an eigenpair of $\h_s$  if and only if there exists an induced subgraph
$\mathcal{G}$ of $\h$ with no isolated vertices having a strictly nonzero eigenpair $(\beta,\mathbf{y})$, where $\lambda^{rs} = \beta^{rs}$.
\end{Lem}

\begin{proof} If $(\lambda,\x)$ is a strictly nonzero eigenpair of $\h_s$ then, by Lemma~\ref{Lema_h_hks} there is a strictly nonzero eigenpair
$(\lambda,\mathbf{y})$  of $\h$. Hence, we may suppose $\x$ has some zero entries. Let $I=\{v\in \h_s:x_v=0\}$. By  Lemma~\ref{LemEigenpair_hk}, we know
that there exists a strictly nonzero eigenpair $(\lambda,\mathbf{z})$ of $\h_s\triangleleft I$. By defining $J$ as the set of all main vertices of
$I$, we observe, first, that $\mathcal{G}=\h\triangleleft J$ is an induced subgraph of $\h$, and second, by Lemma~\ref{LemaPowerRemocao}, that $(\h\triangleleft J)_s= \h_s\triangleleft J$.
Lemma~\ref{LemaTwoVertices} implies that entries of $\x$ associated with vertices of $\h_s$ that are copies of the same vertex of $\h$ differ by a root of unity. In particular, they are either all zero or all nonzero. This implies that $\h_s\triangleleft I=\h_s\triangleleft J$.
So, we have  $\h_s\triangleleft I=\h_s\triangleleft J=(\h\triangleleft J)_s$ and, by Lemma~\ref{Lema_h_hks}, we conclude that there is a
strictly nonzero eigenpair $(\beta,\mathbf{y})$ of $\mathcal{G}= \h\triangleleft J$, where $\lambda^{rs} = \beta^{rs}$.

Conversely, if $(\beta,\mathbf{y})$ is a strictly nonzero eigenpair of an induced subgraph $\mathcal{G}$ of $\h$ with no isolated vertices, then we can write
$\mathcal{G}=\h\triangleleft I$ for some $I\subset V(\h)$.  By Theorem~\ref{TeoKang}, any solution of $\lambda^{rs} = \beta^{rs}$ is an eigenvalue of $(\h\triangleleft I)_s=\h_s\triangleleft I$. Corollary~\ref{Cor_hs} implies that $\lambda$ is an eigenvalue of $\h_s$.
\end{proof}

\begin{Lem}\label{LemaSubgraph_hr}
Let $\h$ be an $r$-graph. The pair $(\lambda,\x)$ is an eigenpair of $\h^{r+1}$ if and only if there exists an induced subgraph $\mathcal{G}$ of $\h$ with no isolated vertices having strictly nonzero eigenpair $(\beta,\mathbf{y})$, where $\beta^{r}=\lambda^{r+1}$.
\end{Lem}

\begin{proof} If $(\lambda,\x)$ is a strictly nonzero eigenpair  of $\h^{r+1}$, then, by Lemma~\ref{Lema_h_hks}, we know that there is a strictly
nonzero eigenpair $(\beta,\mathbf{z})$  of $\mathcal{G}=\h$ such that $\beta^{r}=\lambda^{r+1}$. If $\x$ has some zero entries, we define $I=\{v\in\h^{r+1}:x_v=0\}$ and by the proof of Lemma~\ref{LemEigenpair_hk}, we know that there is a strictly nonzero eigenpair $(\lambda,\mathbf{y})$  of $\h^{r+1}\triangleleft I$. Define $J$ as the set of all main vertices in $I$. Note that $\mathcal{G}=\h\triangleleft J$ is an induced subgraph of $\h$. By
Lemma~\ref{LemaPowerRemocao}, we have that $(\h\triangleleft J)^{r+1}= \h^{r+1}\triangleleft J$.

By Lemma~\ref{LemaAdditionalVertex}, an additional vertex $u$ lies in $I$ if and only if there is a main vertex $v$ in the edge containing $u$ that lies in $I$, so that $\h^{r+1}\triangleleft I=\h^{r+1}\triangleleft J$. By Lemma~\ref{Lema_h_hks}, we conclude that there exists a strictly nonzero eigenpair $(\beta,\mathbf{y})$ of $\mathcal{G}=\h\triangleleft J$ such that $\beta^r= \lambda^{r+1}$.

If $(\beta,\mathbf{y})$ is a strictly nonzero eigenpair  of an induced subgraph $\mathcal{G}$ of $\h$ with no isolated vertex, we can write
$\mathcal{G}=\h\triangleleft I$, for some $I\subset V(\h)$. By Theorem~\ref{TeoKang}, each solution of $\lambda^{r+1} = \beta^r$ is an eigenvalue
of $(\h\triangleleft I)^{r+1}=\h^{r+1}\triangleleft I$, where equality comes from Lemma~\ref{LemaPowerRemocao}(c). By Corollary~\ref{Cor_hr}, we deduce that $\lambda$ is an eigenvalue of $\h^{r+1}$, as required.

\end{proof}

\begin{Lem}\label{LemaSubgraph_hk}
Let $\h$ be an $r$-graph and let $k>r+1$. The pair $(\lambda,\x)$ is an eigenpair of $\h^k$ if and only if some subgraph $\mathcal{G}$ of $\h$ has a strictly
nonzero eigenpair $(\beta,\mathbf{y})$, where $\beta^r=\lambda^k$.
\end{Lem}
\begin{proof} If $(\lambda,\x)$ is a strictly nonzero eigenpair  of $\h^k$, then, by Lemma~\ref{Lema_h_hks}, we know that there is a strictly nonzero
eigenpair $(\beta,\mathbf{y})$ of $\mathcal{G}=\h$. If $\x$ has some zero coordinates, we define  $I=\{v\in V(\h^k):x_v=0\}$, and  by
the proof of Lemma~\ref{LemEigenpair_hk}, there exists a strictly nonzero eigenpair $(\lambda,\mathbf{z})$ of  $\h^k\triangleleft I$. Define $J$ as the
set of all the main vertices in $I$. Let $A$ be the set of edges $e$ of $\h$ such that the entries of $\x$ corresponding to elements of $e$ are nonzero, but such that an entry associated with some additional vertex in $S_e$ is zero. By Lemma~\ref{LemaTwoVertices}, this means that that the entries associated with all elements of $S_e$ are zero. Note that $\h^k\triangleleft I=(\h^k\triangleleft J) - A^k$. By Lemma~\ref{LemaPowerRemocao} we have that
\[\h^k\triangleleft I=(\h^k\triangleleft J) - A^k=(\h\triangleleft J)^k - A^k=((\h\triangleleft J) - A)^k\]
Thus, by Lemma~\ref{Lema_h_hks}, we know that there is an eigenpair $(\beta,\mathbf{y})$ of $\mathcal{G}=(\h\triangleleft J) - A$ such that $\beta^r=\lambda^{k}$.

Let $(\beta,\mathbf{y})$ be a strictly nonzero eigenpair of a subgraph $\mathcal{G}$ of $\h$. Consider a set of vertices $I$ and a set of edges $A$ whose elements do not contain vertices of $I$ such that $\mathcal{G}=(\h \triangleleft I)-A$. By Lemma~\ref{LemaPowerRemocao}, we have that
\[(\mathcal{G})^k=((\h\triangleleft I)-A)^k=(\h\triangleleft I)^k-A^k=(\h^k\triangleleft I)-A^k\]
Let $J \subset V(\h^k)$ given by the union of the set $I$ with the additional vertices of the $k$-edges in $A^k$. The definition of $J$ implies that $\h^k\triangleleft J=(\h^k\triangleleft I)-A^k$. By Theorem~\ref{TeoKang}, each solution of $\lambda^k=\beta^r$ is an eigenvalue of $\h^k\triangleleft J$. Corollary~\ref{Cor_hk} implies that $\lambda$ is an eigenvalue of $\h^k$.
\end{proof}

To conclude the paper, we show how the above results imply Theorem~\ref{Main}.
\begin{proof}[Proof of Theorem~\ref{Main}]
Let $\h$ be an $r$-graph, fix integers $s \geq 1$ and $k \geq rs$ and consider a pair $(\lambda,\x) \in \mathbb{C} \times \mathbb{C}^{n}$ such that $\lambda \neq 0$ and $n=(k-rs) \cdot e(\h)+s \cdot v(\h)$.

The statements of Lemmas~\ref{LemaSubgraph_hs},~\ref{LemaSubgraph_hr} and~\ref{LemaSubgraph_hk} refer only to strictly
nonzero eigenpairs of subgraphs $\mathcal{G}$ of the base hypergraph $\h$, while Theorem~\ref{Main} refers to all nonzero eigenvalues. Because of Lemma~\ref{LemEigenpair_hk}, we know that $(\lambda,\x)$ is an eigenpair of an $r$-graph $\mathcal{G}$ if and only if there is a subgraph $\mathcal{G}'$ of $\mathcal{G}$ such that $(\lambda,\y)$ is a strictly nonzero eigenpair of $\mathcal{G}'$. In particular, in the proof of Theorem~\ref{Main}, we may restrict our attention to strictly nonzero eigenpairs of subgraphs (or induced subgraphs) of $\h$.

We start with part (a). If $s\geq 2$ and $k=rs$, the result is just Lemma~\ref{LemaSubgraph_hs}. Assume that $k=rs+1$. If $s=1$, the theorem is just Lemma~\ref{LemaSubgraph_hr}, so suppose that $s \geq 2$.  Let $(\lambda,\x)$ be an eigenpair of $\h^{k}_s$. By Lemma~\ref{LemaSubgraph_hr}, there exists an induced subgraph $\mathcal{G}^\ast$ of $\h_s$ with no isolated vertices having strictly nonzero eigenpair $(\beta,\y)$, where $\beta^{rs}=\lambda^{rs+1}$. In fact, the proof of Lemma~\ref{LemaSubgraph_hr} implies that $\mathcal{G}^\ast$ is equal to $\mathcal{G}_{s}$ for some induced subgraph $\mathcal{G}$ of $\h$ (this is a consequence of Corollary~\ref{cor_useful}). The desired result is obtained by an application of Lemma~\ref{LemaSubgraph_hs} to $\mathcal{G}$  and $\mathcal{G}_s$.
Conversely, assume that $(\beta,\y)$ is a strictly nonzero eigenpair of and induced subgraph $\mathcal{G}$ of $\h$ with no isolated vertices, where $\beta^{rs}=\lambda^k$. By Lemma~\ref{LemaSubgraph_hs}, the quantity $\beta$ is an eigenvalue of $\h_s$. We may now apply Theorem~\ref{TeoKang} for $r'=rs$, $k'=r'+1$ and $\h'=\h_s$ to conclude that the solutions to the equation $\lambda^{rs+1}=\beta^{rs}$ are eigenvalues of $\h^k_s$, leading to the desired result.

Part (b) may be obtained analogously replacing Lemma~\ref{LemaSubgraph_hr} by Lemma~\ref{LemaSubgraph_hk}.
\end{proof}

To conclude our paper, we provide examples that shed light on aspects of our main result. First, we use Theorem~\ref{Main} to compute all nonzero eigenvalues of a power hypergraph. Let $k\geq 3$ and suppose that we wish to find the nonzero eigenvalues of $(S_n)^k$, where $S_n$ denotes the star with $n$ vertices (i.e., the complete bipartite graph $K_{1,n-1}$). It is well-known that the only nonzero eigenvalues of any star $S_m$ are $\sqrt{m-1}$ and $-\sqrt{m-1}$. Moreover, in the case of stars, a subgraph (or induced subgraph) with no isolated vertex is always a star $S_m$ with $2 \leq m \leq n$. By Theorem~\ref{Main} the nonzero eigenvalues of $(S_n)^k$ are the complex roots of the equations $x^k=p$, where $p \in \{1,2,\ldots,n-1\}$.

Next, we illustrate the difference between parts (a) and (b) of
Theorem~\ref{Main}. Consider the graph $\h=C_4$, so that $\mathcal{G}=P_4$ is
a subgraph of $\h$ with no isolated vertices. It is well-known that $\nu =
\frac{1+\sqrt{5}}{2}$ is an eigenvalue of $P_4$. This means that the complex
roots of the equation $x^k=\nu^2$ are nonzero eigenvalues of $\h^k$ for all
$k \geq 4$. As it turns out, the roots of $x^3=\nu^2$ are not eigenvalues of
$\h^3$, but this does not contradict Theorem~\ref{Main} because $P_4$ is not
an induced subgraph of $C_4$.

\section*{Acknowledgments} This work is part of
doctoral studies of K. Cardoso under the supervision of V.~Trevisan. K.
Cardoso is grateful for the support given by Instituto Federal do Rio Grande
do Sul (IFRS), Campus Feliz. C. Hoppen thanks CNPq (Proj. 308539/2015-0) for
their support. V. Trevisan acknowledges partial support of CNPq grants
409746/2016-9 and 303334/2016-9, CAPES (Proj. MATHAMSUD 18-MATH-01) and
FAPERGS (Proj.\ PqG 17/2551-0001).

\vspace{-0.01cm}
%%%%%%%%%%%%%%%%%%%%%%%%%%%%%%%%%%%%%%%%%%%%%%%%%%%%%%%%%%%%%%%%%%%%%%%%%%%%%%%%%%%%%%%

\end{document}